\documentclass[11pt]{article}

\usepackage{amssymb,amsmath,amsthm}
\usepackage{amsfonts}
\usepackage{latexsym}
\usepackage{enumerate}
\usepackage{wasysym}
\usepackage{fancyhdr}
\usepackage{lipsum}
\usepackage{epsfig}
\usepackage{breqn}
\usepackage{footnpag}
\usepackage{rotating}
\usepackage{amsfonts}
\usepackage{setspace}
\usepackage{fullpage}
\usepackage{enumitem}
\usepackage{bbold} 
\usepackage{comment}
\usepackage{pgf,tikz}
\usepackage{hyperref}
\usepackage{verbatim}
\usepackage{color}
\usepackage{mathrsfs}
\usepackage{soul}
\usetikzlibrary{arrows}
\usepackage{array}

\newtheoremstyle{dotless}{}{}{\itshape}{}{\bfseries}{}{ }{}

\theoremstyle{dotless}

\newtheorem{theorem}{Theorem}[section] %numbering divided by section
\DeclareMathOperator{\ex}{ex}
\DeclareMathOperator{\sat}{sat}

\def\E{\mbox{{\bf E}}}

\title{Nearly-Regular Hypergraphs and Saturation of Berge Stars}
\author{Bethany Austhof and Sean English}

\begin{document}
	\maketitle
	
\begin{abstract}
	Given a graph $G$, we say a $k$-uniform hypergraph $H$ on the same vertex set contains a Berge-$G$ if there exists an injection $\phi:E(G)\to E(H)$ such that $e\subseteq\phi(e)$ for each edge $e\in E(G)$. A hypergraph $H$ is Berge-$G$-saturated if $H$ does not contain a Berge-$G$, but adding any edge to $H$ creates a Berge-$G$. The saturation number for Berge-$G$, denoted $\sat_k(n,\text{Berge-}G)$ is the least number of edges in a $k$-uniform hypergraph that is Berge-$G$-saturated. We determine exactly the value of the saturation numbers for Berge stars. As a tool for our main result, we also prove the existence of nearly-regular $k$-uniform hypergraphs, or $k$-uniform hypergraphs in which every vertex has degree $r$ or $r-1$ for some $r\in \mathbb{Z}$, and less than $k$ vertices have degree $r-1$. 
\end{abstract}

\section{Introduction}
The main problem in extremal graph theory involves finding the \emph{extremal number} of $F$, $\ex(n,F)$, which is the maximum number of edges among all $n$ vertex graphs that do not contain a subgraph isomorphic to some forbidden graph $F$. This problem was originally studied by Mantel for triangles, \cite{M07}. Tur\'an's Theorem generalized this, giving the value of $\ex(n,K_s)$ for all $s$, \cite{T41}.

We say a graph $G$ is $F$-free if $G$ does not contain a subgraph isomorphic to $F$. An easy but interesting observation is that if $G$ is an $F$-free on $n$ vertices with $|E(G)|=\ex(n,F)$, then $G$ has the property that for any edge $e$ in the complement of $G$, $e\in E(\overline{G})$, adding $e$ to $G$ must create a subgraph isomorphic to $F$. This leads to the following natural definition: We say $G$ is $F$-saturated if $G$ is $F$-free, but for any edge $e\in E(F)$, $G+e$ contains a copy of $F$. Thus, we can say that $\ex(n,F)$ is the maximum number of edges in any $F$-saturated graph on $n$ vertices. This leads to an interesting minimization problem associated with extremal numbers.

The saturation number of a forbidden graph $F$, denoted $\sat(n,F)$ is the least number of edges over all graphs $G$ on $n$ vertices that are $F$ saturated. It has been seen that saturation numbers and extremal numbers behave very differently. Possibly the most striking difference is in their asymptotic growth rates. The Erd\H{o}s-Stone Theorem, sometimes referred to as the Fundamental Theorem of Extremal Graph Theory, characterizes the growth rate of extremal numbers for all non-bipartite forbidden graphs. Let $\chi(F)$ denote the chromatic number of $F$. Given functions $f=f(n)$ and $g=g(n)$, we write $f=O(g)$ if there exists some constant $c$ such that $f\leq c g$ for all sufficiently large $n$, and we write $f=o(g)$ if $\lim_{n\to\infty} f/g=0$.

\begin{theorem}{Erd\H{o}s-Stone Theorem, \cite{ES46}}
	For all non-empty forbidden graphs $F$, we have
	\[
	\ex(n,F)=\left(\frac{\chi(F)-2}{\chi(F)-1}-o(1)\right)\binom{n}2
	\]
\end{theorem}

Thus, extremal numbers for non-bipartite forbidden graphs grow quadratically in $n$. In contrast to this, we have the following theorem by K\'asonyi and Tuza, which shows that saturation numbers grow no faster than linearly in $n$.
\begin{theorem}{\cite{KT86}}
	For all forbidden graphs $F$, we have
	\[
	\sat(n,F)=O(n).
	\]
\end{theorem}

Extremal numbers and saturation numbers have also been studied for hypergraphs. A hypergraph $H$ is a generalization of a graph, where the edges of $H$ can contain arbitrarily many vertices, rather than just two. A hypergraph is called $k$-uniform if every edge contains exactly $k$ vertices. Thus, a $2$-uniform hypergraph is just a graph. Hypergraph extremal problems are notoriously difficult, for example, let $K_4^{(3)}$ denote the complete $3$-uniform hypergraph on $4$ vertices. Not even the growth rate of $\ex(n,K_4^{(3)})$ is known, even though this may be the easiest non-trivial hypergraph to look at.

While in general hypergraph extremal problems have been too difficult to make much progress on, recently specific interesting families of hypergraphs have been studied, and significant progress has been made for these families. Given a graph $F$ and hypergraph $H$ embedded on the same vertex set, we say a hypergraph $H$ is Berge-$F$ if there is a bijection $\phi:E(F)\to E(H)$ such that $e\subseteq \phi(e)$ for all $e\in E(F)$. This can be thought of as adding vertices to the edges of $F$ to make them hyperedges that form a copy of $H$, or shrinking down the hyperedges of $H$ to graph edges that create $F$. It is worth noting that many non-isomorphic hypergraphs can be Berge-$F$ for the same graph $F$.

Analogously to the graph case, we can say a $k$-uniform hypergraph $H$ is Berge-$F$-saturated if $H$ does not contain a subgraph isomorphic to a Berge-$F$, but adding any hyperedge to $H$ creates a copy of Berge-$F$. Based on this, we can define the $k$-uniform extremal number $\ex_k(n,\text{Berge-}F)$ and saturation number $\sat_k(n,\text{Berge-}F)$ to be the maximum, and respectfully minimum, number of edges in a Berge-$F$-saturated $k$-uniform hypergraph on $n$ vertices. Extremal numbers for Berge hypergraphs have been studied extensively, \cite{lazebnik2003hypergraphs, gyHori2006triangle, bollobas2008pentagons, gyHori20123, furedi20173, T17, gerbner2017asymptotics}. On the other hand, saturation numbers for Berge hypergraphs have been mostly left untouched.

In the seminal paper on saturation numbers for Berge hypergraphs, saturation number for Berge hypergraphs for many common classes of graphs were studied by the second author and others, including triangles, matchings, cycles, paths and stars \cite{EGKMS2017}. The growth rate of saturation numbers also have been studied by the second author and others, and it has been determined that $\sat_k(n,\text{Berge-}F)=O(n)$ for $3\leq k\leq 5$ \cite{EGMT2018}. Recently, Axenovich and Winter have begun considering Berge saturation for non-uniform hypergraphs, showing that there are Berge-$F$-saturated non-uniform hypergraphs with $|E(F)|-1$ edges for all graphs $F$ except stars, in which there are saturated examples with $|E(F)|$ edges \cite{AW2018}. Here we will study the saturation number for Berge stars in the uniform case. A special case of stars was considered by the second author and others.

\begin{theorem}{\cite{EGKMS2017}}\label{theorem star from old paper}
For all $n\geq k^2$, 
\[
\sat_k(n,K_{1,k+1})=n-k+1.
\]
\end{theorem}

Due to a certain structure necessary in the proof of the preceding theorem, it only applies to stars where the number of leaves is exactly one greater than the uniformity of the host hypergraph. In this note, we will extend this result by determining these saturation numbers exactly for any uniformity and any number of leaves. Here we present our main result

\begin{theorem}\label{theorem main}
	For all $k\geq 3$, $\ell\in \mathbb{N}$, and large $n$, we have
	\[
	\sat_k(n,\text{Berge-}K_{1,\ell})=\min_{a\in[n],\binom{a-1}{k-1}\leq \ell-2} \left\lceil\frac{(\ell-1)(n-a)}k\right\rceil+\binom{a}k.
	\]
\end{theorem}

Similarly to the proof of Theorem \ref{theorem star from old paper} from \cite{EGKMS2017}, our main result involves finding a hypergraph with certain structural properties. A hypergraph $H$ is linear if for every pair of edges $e,f\in E(H)$, $|e\cap f|\leq 1$. Further, a $k$-uniform hypergraph is nearly-$d$-regular if every vertex has degree (total number of edges that contain that vertex) either $d$ or $d-1$, and less than $k$ vertices have degree $d-1$. If $dn/k$ is an integer, then a nearly-$d$-regular hypergraph will have only vertices of degree $d$, so we just say the hypergraph is $d$-regular. For Theorem \ref{theorem main}, we need nearly-$d$-regular $k$-uniform hypergraphs that are also linear.

A linear $k$-uniform $d$-regular hypergraphs on $n$ vertices with $m=\frac{dn}k$ edges is equivalent to an incidence structure known as a $(n,m,r,k)$-configuration, which is a set of $n$ points and $m$ lines such that each line contains $k$ points, each point is contained in $d$ lines, and lines intersect in at most one point.

It is known that for $n$ large enough, as long as $dn/k$ is an integer, $(n,m=\frac{dn}k,d,k)$-configurations exist, \cite{BS12}, and thus so do linear $d$-regular $k$-uniform hypergraphs. For the purposes of this paper though, linear nearly-$d$-regular $k$-uniform hypergraphs (i.e. we may not have $dn/k\in \mathbb{Z}$) are necessary. This is in essence a problem of graphical degree sequences. 

Given a finite sequence of non-negative integers, $d_1\geq d_2\geq\dots\geq d_n$, we say the sequence is graphical if there exists a simple graph on $n$ vertices whose degree sequence matches this sequence. The Erd\H{o}s-Gallai theorem, \cite{EG59}, gives an efficient characterization of graphical degree sequences. Unfortunately there is not an analogous result known for $k$-uniform hypergraphs. 

A sequence of non-negative integers is called $k$-graphical if there exists a $k$-uniform hypergraph with that degree sequence, and is called linearly-$k$-graphical if there exists a linear $k$-uniform hypergraph with the desired degree sequence. For recent work on $k$-grahical degree sequences, see \cite{BEFHRST2013}. In \cite{BBD2008}, the authors provide an Erd\H{o}s-Gallai-type theorem for linear hypergraphs, but only for non-uniform hypergraphs, so their results do not apply to nearly-regular uniform hypergraphs. Thus, as a tool for determining the saturation number for Berge stars, we also prove the existence of nearly-regular hypergraphs.

\begin{theorem}\label{theorem existence}
	Let $d\geq 1$ and $k\geq 2$. Then for all sufficiently large $n$, there exists a nearly-$d$-regular $k$-uniform hypergraph on $n$ vertices.
\end{theorem}

In order to prove the preceding result, we will use the probabilistic method. The main idea behind the probabilistic method is that if one can exhibit a probabilistic experiment that has a positive probability of outputting a nearly-$d$-regular $k$-uniform hypergraph on $n$ vertices, then such a structure must necessarily exist.

To do this, we will use the configuration model for hypergraphs. Discussed in detail for graphs in \cite{JLR}, the configuration model for hypergraphs produces uniformly at random a $k$-uniform pseudo-hypergraph (a $k$-uniform hypergraph that may have repeated edges, and with edges that contain the same vertex multiple times) with a prescribed degree sequence. We will show that this model has a positive probability of producing a linear nearly-regular uniform simple hypergraph.

The layout of the rest of the paper is as follows: In Section \ref{section configuration model}, the configuration model will be discussed in more detail. In Section \ref{section existence}, the configuration model is used to prove Theorem \ref{theorem existence}. Finally, in section \ref{section Berge stars} the main theorem, Therem \ref{theorem main} is proved using Theorem \ref{theorem existence}. 

\section{Linear Nearly-Regular Uniform Hypergraphs}

\subsection{The Configuration Model}\label{section configuration model}

The following random model was first used implicitly by Bender and Canfield \cite{BC78} and made explicit by Bollob\'as \cite{B80}. 

Let $k\geq 2$ be an integer and $\mathbf{d}=(d_1,d_2,\dots,d_n)$ be a sequence of non-negative integers such that $k\mid \sum_{i=1}^nd_i$. We will describe how to generate a random $k$-uniform pseudo-hypergraph on $n$ vertices with degree sequence $\mathbf{d}$.

Let $S=\{v_{i,j}\mid 1\leq i\leq n, 1\leq j\leq d_i\}$ be set of $\sum_{i=1}^nd_i$ elements, which we will call \emph{configuration points}. For each $1\leq i\leq n$, let $V_i=\{v_{i,j}\mid 1\leq j\leq d_i\}$. A \emph{configuration} is a $k$-uniform perfect matching $M$ with vertex set $S$. To each configuration, we can associate a random pseudo-hypergraph $H$ with degree sequence $\mathbf{d}$: Let $V(H)=\{V_1,V_2,\dots,V_n\}$ and for each $k$-edge $e\in M$, $e=\{v_{i_1,j_1},v_{i_2,j_2},\dots,v_{i_k,j_k}\}$ add the $k$-element multiset $\{V_{i_1},V_{i_2},\dots,V_{i_k}\}$ to $E(H)$. This can be thought of as taking the matching $M$, and collapsing all the configuration points in $V_i$ down to a single vertex for each $1\leq i\leq n$, while preserving adjacencies. Let $\mathbb{H}^{(k)}_*(n,\mathbf{d})$ denote the probability space whose outcome is the $k$-uniform pseudo-hypergraph associated with a configuration chosen uniformly at random. Let $\phi(x)$ denote the number of configurations on $x$ points, and note that 
\begin{equation}\label{equation number of configurations}
\phi(x)=\frac{x!}{(k!)^{x/k}(x/k)!}.
\end{equation}

This process may create loops, which is when some edge in $M$ intersects some $V_i$ in more than one point. More formally, we will say a pair of configuration points $v_{i,j_1},v_{i,j_2}$ form a \emph{loop} at $V_i$ if they are contained in the same edge in $M$. For example, if four configuration points from the same set $V_i$ ended up all together in a single edge of $M$, we will count this as being $\binom{4}2=6$ loops at $V_i$. We would like to construct a linear hypergraph, so we are also interested in the number of edges this process creates that overlap in two or more vertices. More formally, we will say that four configuration points $v_{i_1,j_1},v_{i_1,j_2},v_{i_2,j_3},v_{i_2,j_4}$ form an \emph{overlap} if there exist two $k$-edges $e,f\in M$ such that $v_{i_1,j_1},v_{i_2,j_3}\in e$ and $v_{i_1,j_2},v_{i_2,j_4}\in f$, or $v_{i_1,j_1},v_{i_2,j_4}\in e$ and $v_{i_1,j_2},v_{i_2,j_3}\in f$. Note that if a configuration has no loops and no overlaps, then the associated hypergraph is a simple linear hypergraph.

\subsection{The Existence of Linear Nearly-Regular Uniform Hypergraphs}\label{section existence}

We will use the method of moments and the configuration model to show that linear nearly-regular hypergraphs exist. More precisely, we will need the following theorem. Here, given an integer $X$, let $(X)_t=\prod_{i=0}^{t-1} (X-i)=\frac{X!}{(X-t)!}$ denote the falling factorial.

\begin{theorem}{(Theorem 6.10 in \cite{JLR}}\label{theorem method of moments}
 Let $X_1,X_2,\dots,X_n,\dots$ and $Y_1,Y_2,\dots,Y_n,\dots$ be two sequences of random variables. If $\lambda,\mu\geq 0$ are real numbers such that, as $n\to \infty$, we have
 \[
 \E\left[ (X_n)_{\ell_1}\cdot(Y_n)_{\ell_2}\right]\to \lambda^{\ell_1}\cdot \mu^{\ell_2}
 \]
 for all integers $\ell_1,\ell_2\geq 0$, then $X_n$ and $Y_n$ converge in distribution to independent Poisson random variables with mean $\lambda$ and $\mu$ respectively.
\end{theorem}

We now have everything we need to show the existence of nearly-$d$-regular linear hypergraphs. The proof of this is a straightforward generalization of the proof provided in \cite{JLR} on the number of small cycles of different lengths in random regular graphs. For the sake of completeness, we provide the details of this generalization here.

\begin{proof}{proof of Theorem \ref{theorem existence}}
	
Given integers $n\geq d$ with $d$ constant, let $r=(dn\mod k)$, and let $\mathbf{d}$ be the degree sequence of a $n$ vertex nearly-$d$-regular hypergraph. Let $H=\mathbb{H}^{(k)}_{*}(n,\mathbf{d})$ be an outcome of the configuration model. Note that we have a total of $nd-r$ configuration points. Let $Z_1$ and $Z_2$ be the random variables that tracks the number of loops and the number of overlap in $H$ respectively. Our goal is to show that we have $Z_1=Z_2=0$ with positive probability, which will imply the existence of the desired hypergraph. To accomplish this, we will actually prove something much stronger; using the method of moments, we will show that $Z_1$ and $Z_2$ converge to independent Poisson random variables. Since Poisson random variables have a positive probability of being $0$, this will complete the proof.

Towards applying Theorem \ref{theorem method of moments}, let $\lambda=\frac{(d-1)(k-1)}2$ and $\mu=\left(\frac{(d-1)(k-1)}2\right)^2$, and fix integers $\ell_1,\ell_2\geq 0$. Consider the random variable $X=(Z_1)_{\ell_1}(Z_2)_{\ell_2}$. This counts ordered pairs, where the first coordinate contains an ordered set of $\ell_1$ distinct loops in $H$ and the second coordinate contains an ordered set of $\ell_2$ distinct overlapping pairs in $H$. a collection of $\ell_1$ loops and $\ell_2$ overlaps involves at most $\ell_1+2\ell_2$ edges in $M$, at most $\ell_1+2\ell_2$ vertices $V_i$, and at most $2\ell_1+4\ell_2$ configuration points $v_{i,j}$. This happens when each loop and overlap in question are in distinct edges and with distinct vertices. We will show that if this is not the case, the contribution to $\E(X)$ is negligible.

We will say a collection of $\ell_1$ loops and $\ell_2$ overlaps is of type $(a,b,c)$ if the collection involves $a$ edges, $b$ vertices and $c$ configuration points. Let $Y$ be the random variable that counts contributions to $X$ from collections of loops and overlaps of type $(a,b,c)$ when one or more of $a,b,c$ are not at their maximum value. Recall that $\phi(x)$, from Equation \eqref{equation number of configurations} is the function that counts the number of $k$-edge matchings on $x$ configuration points. For ease of notation set $C_1=(db)_c(ak)!$, and note that $C_1$ is constant with respect to $n$. Then
\begin{align*}
\E(Y)&\leq \sum_{\begin{array}{c}(a,b,c)\in [\ell_1+2\ell_2]^2\times [2\ell_1+4\ell_2], \\ (a,b,c)\neq (\ell_1+2\ell_2,\ell_1+2\ell_2,2\ell_1+4\ell_2)\end{array}}C_1\binom{n}b\binom{nd-r-c}{ak-c}\frac{\phi(nd-r-ak)}{\phi(nd-r)}\\
&=\sum_{\begin{array}{c}(a,b,c)\in [\ell_1+2\ell_2]^2\times [2\ell_1+4\ell_2], \\ (a,b,c)\neq (\ell_1+2\ell_2,\ell_1+2\ell_2,2\ell_1+4\ell_2)\end{array}}C_1\binom{n}b\binom{nd-r-c}{ak-c}\frac{(k!)^a\left(\frac{nd-r}k\right)_a}{(nd-r)_{ak}}\\
&=\sum_{\begin{array}{c}(a,b,c)\in [\ell_1+2\ell_2]^2\times [2\ell_1+4\ell_2], \\ (a,b,c)\neq (\ell_1+2\ell_2,\ell_1+2\ell_2,2\ell_1+4\ell_2)\end{array}}O(n^{a+b-c}).
\end{align*}
The explanation for the the first line of the preceding inequality is as follows. If we fix $a,b,c$, we can chose the $b$ vertices involved in $\binom{n}b$ ways. There are at most $db$ configuration points in these $b$ vertices, so we can chose and order the $c$ configuration points in at most $(db)_c$ ways. The ordering of these vertices gives an overcount of the number of ways we can chose which of these $c$ configuration points belong to which of the loops and overlap pairs. We then choose the remaining $ak-c$ vertices in the $a$ edges, and then $(ak)!$ overcounts the number of ways we can distribute the configuration points into the edges. Finally, $\phi(nd-r-ak)$ counts the number of ways to choose the remaining edges in the matching.

From the preceding inequality, if we show that $c>a+b$, we have $\E(Y)=o(1)$. Consider the $\ell_1+\ell_2$ pairs of vertices in our loops and quadruples of vertices in our overlapping edges one at a time, and as we do, we will mark each unmarked edge, vertex and configuration point involved with the pair. Each time we add a loop, if one of the configuration points in the loop is already marked, then both the edge and the vertex involved with the loop must have already been marked as well. Thus, each loop marks at least as many configuration points as vertices and edges. Similarly, if we add a quadruple involved with overlapping edges, if one configuration point was already marked, then both an edge and a vertex involved has already been marked. If two configuration points were already marked, then actually at least three edges and vertices must have already been marked (either two edges and one vertex, or two vertices and one edge, depending). If three or four configuration points were already marked, then all four of the edges and vertices involved in the overlap were already marked. In any case, we always mark at least as many cluster points as we do edges and vertices, so $c\geq a+b$. 

To see strict inequality, it suffices to note that if $c=2\ell_1+4\ell_2$, by the fact that $(a,b,c)\neq (\ell_1+2\ell_2,\ell_1+2\ell_2,2\ell_1+2\ell_2)$ gives the result, and if $c<2\ell_1+4\ell_2$, there must have been a first loop or pair of overlapping edges in which there was already a marked configuration point, say $v_{i,j}$. If before this there was never a case where we ran into a pre-marked vertex or edge, then the second configuration point $v_{i,j^*}$ in the same vertex as $v_{i,j}$ must not have been marked, but both the edge and the vertex that $v_{i,j}$ is in were pre-marked. If this was a loop, then we are done. If it was a quadruple involved in an overlap, then in the vertex that $v_{i,j}$ is not in, we must also have at most one pre-marked configuration point, and if so, the second vertex involved was also pre-marked. Thus, the first time we encounter a pre-marked configuration point, it must be that either there is one pre-marked configuration point and two pre-marked edges and vertices, or there are two pre-marked configuration points and at least three pre-marked edges and vertices. In either case, this is enough to guarantee $c>a+b$. Thus $\E(Y)=o(1)$.

Now, let $S$ be the contribution to $X-Y$ in which at least one vertex with a loop or involved in an overlap is degree $d-1$. We will show $\E(S)=o(1)$. Let $\ell=\ell_1+2\ell_2$, $m=\min\{\ell_1+2\ell_2,r\}$ and $C_2=(d\ell)_{2\ell}(k\ell)!$. We have

\begin{align*}
\E(S)&\leq \sum_{i=1}^mC_2\binom{r}{i}\binom{n-r}{\ell-i}\binom{nd-r-2\ell}{(k-2)\ell}\frac{\phi(nd-r-k\ell)}{\phi(nd-r)}\\
&=\sum_{i=1}^mC_2\binom{r}{i}\binom{n-r}{\ell-i}\binom{nd-r-2\ell}{(k-2)\ell}\frac{(k!)^\ell\left(\frac{nd-r}k\right)_\ell}{(nd-r)_{\ell k}}=\sum_{i=1}^mO(n^{-i})=o(1).
\end{align*}
Indeed, we first choose how many vertices will be of degree $d-1$, then chose the $\ell$ vertices involved in all the loops and overlapping pairs. The constant $(d\ell)_{2\ell}$ overcounts the number of ways to choose and order the $2\ell$ configuration points involved in loops and overlaps, and the ordering overcounts how many ways we can choose which configuration points belong to which loops and pairs. Then we choose the remaining $(k-2)\ell$ configuration points involved in the $\ell$ edges that contain loops and overlaps. The constant factor $(k\ell)!$ then gives an ordering of these vertices, which overcounts the number of ways the $k\ell$ configuration points can be sorted into the $\ell$ edges.

Let $X^*=X-Y-S$. We now assume that all the loops and overlaps occur in vertices of degree $d$, and each loop and overlap occur with different configuration points, on all different vertices, with all different edges. Recall that $\lambda=\frac{(d-1)(k-1)}2$ and $\mu=\left(\frac{(d-1)(k-1)}2\right)^2$. Here we get

\begin{align*}
\E(X^*)&=\frac{(n-r)_{\ell}}{2^{\ell_2}}\binom{d}2^\ell\binom{nd-r-2\ell}{(k-2)\ell}2^{\ell_2}\frac{((k-2)\ell)!}{((k-2)!)^\ell}\frac{\phi(nd-r-k\ell)}{\phi(nd-r)}\\
&=(1+o(1))n^\ell\left(\frac{d(d-1)}2\right)^\ell\frac{(nd-r-2\ell)_{(k-2)\ell}}{((k-2)\ell)!}\frac{((k-2)\ell)!}{((k-2)!)^\ell}\frac{(k!)^\ell\left(\frac{nd-r}k\right)_\ell}{(nd-r)_{\ell k}}\\
&=(1+o(1))n^\ell\left(\frac{d(d-1)}2\right)^\ell(nd)^{(k-2)\ell}\frac{(k(k-1))^\ell\left(\frac{nd}k\right)^\ell}{(nd)^{\ell k}}\\
&=(1+o(1))\left(\frac{(d-1)(k-1)}{2}\right)^\ell=(1+o(1))\lambda^{\ell_1}\cdot\mu^{\ell_2}
\end{align*}

Since $X$ counts ordered pairs of ordered sets of loops and overlaps, the factor $\frac{(n-r)_{\ell}}{2^{\ell_2}}$ chooses which vertices are involved in the loops and overlaps, and orders then, while the corrective term in the denominator accounts for the fact that overlaps involve two unordered vertices. Then the power of $\binom{d}2$ chooses which configuration points are in the $\ell$ loops and overlaps. We then choose the remaining $(k-2)\ell$ configuration points involved in the $\ell$ edges. Now, we sort the configuration points into edges. First, for each overlap with configuration points $v_{i_1,j_1},v_{i_1,j_2},v_{i_2,j_3},v_{i_2,j_4}$, we need to choose if $v_{i_1,j_1}$ is in an edge with $v_{i_2,j_3}$ or $v_{i_2,j_4}$, giving us a factor of $2^{\ell_2}$. Then we choose an ordered $(k-2)$-matching on the $(k-2)\ell$ configuration points in $\frac{((k-2)\ell)!}{((k-2)!)^\ell}$ ways. The ordering given here gives a pairing between the pairs of vertices in overlaps and loops and the $(k-2)$-edges in the matching, which gives us our $\ell$ edges. Then the final term counts how many ways we can put a $k$-matching down on the rest of the configuration points.

Now, since $\E(X)=\E(X^*)+o(1)$, the conditions of Theorem \ref{theorem method of moments} are met, so we have that $Z_1$ and $Z_2$ converge to independent Poisson random variables with mean $\lambda$ and $\mu$ respectively. This implies that $\Pr(Z_1=Z_2=0)=(1+o(1))e^{-(\lambda+\mu)}$, so for large enough $n$, there is a positive probability of $\mathbb{H}_*^{(k)}(n,\mathbf{d})$ producing a simple linear hypergraph, finishing the proof.
\end{proof}

\section{Saturation of Berge Stars}\label{section Berge stars}

To determine the saturation number for Berge stars, we need to give a few definitions. Let $F$ be a graph. Then a vertex $v$ in some Berge-$F$ is called a \emph{core vertex} if there exists a way to shrink the edges of Berge-$F$ down to create a copy of $F$ that contains the vertex $v$. When we consider a Berge-$K_{1,\ell}$, we will say a core vertex corresponding to a leaf of $K_{1,\ell}$ is a \emph{core leaf}.

\begin{theorem}
	For all $k\geq 3$, $\ell\in \mathbb{N}$, and large $n$, we have
	\[
	\sat_k(n,\text{Berge-}K_{1,\ell})=\min_{a\in[n],\binom{a-1}{k-1}\leq \ell-2} \left\lceil\frac{(\ell-1)(n-a)}k\right\rceil+\binom{a}k.
	\]
\end{theorem}

\begin{proof}
	First we will establish the lower bound. Let $H$ be a $k$-uniform Berge-$K_{1,\ell}$ saturated hypergraph on $n$ vertices. Let $A\subseteq V(H)$ be the set of vertices with degree less than $\ell-1$. Note that $H[A]=K^{(k)}_{|A|}$ since if any $k$ vertices in $A$ are not in an edge together, adding this edge cannot create a Berge-$K_{1,\ell}$. This implies that $\binom{|A|-1}{k-1}\leq \ell-2$ since the vertices of $|A|$ have degree $\leq \ell-2$.
	
	Now we can count the number of edges in $H$ that are not completely contained in $A$. Since the vertices in $B=V(H)\setminus A$ all have degree at least $\ell-1$, we have the following:
	\[
	\sum_{e\in E(H)} |e\cap B|\geq (\ell-1)|B|.
	\]
	If $e\subseteq A$, then $|e\cap B|=0$, and otherwise, $|e\cap B|\leq k$, so
	\[
	\left(|E(H)|-\binom{|A|}k\right) k\geq (\ell-1)|B|,
	\]
	so
	\[
	|E(H)|\geq \left\lceil\frac{(\ell-1)|B|}{k}\right\rceil+\binom{|A|}k.
	\]
	Since $|B|=n-|A|$ and $|A|\in [n]$,  the lower bound follows.
	
	Now, let us consider the upper bound. We will give a construction that is Berge-$K_{1,\ell}$-saturated with the correct number of edges. Let $c$ be such that 
	\[
	\min_{a\in[n],\binom{a-1}{k-1}\leq \ell-2} \left\lceil\frac{(\ell-1)(n-a)}k\right\rceil+\binom{a}k=\left\lceil\frac{(\ell-1)(n-c)}k\right\rceil+\binom{c}k.
	\]
	Let $|V|=n$. Let $C\subseteq V$ be such that $|C|=c$. First, add all the edges in $\binom{C}k$. Now, construct a $k$-uniform nearly-$(\ell-1)$-regular linear hypergraph on $V\setminus C$. We know such a structure exists for large enough $n$ by Theorem \ref{theorem existence}. Let $D\subseteq V\setminus C$ be the set of vertices in the nearly-regular hypergraph that have degree $\ell-2$. If $D$ is empty, the construction is done. Otherwise, add one more edge containing $D$ and $k-|D|$ vertices from $C$. Note that $C$ is large enough for this since the fact that $\binom{a}k=0$ for $1\leq a\leq k-1$, and the fact that  $\left\lceil\frac{(\ell-1)(n-a)}k\right\rceil$ is strictly decreasing with $a$ implies that $c\geq k-1\geq k-|D|$. This completes the construction. It is clear from the construction that this hypergraph has the desired number of edges.
	
	Note that the construction is Berge-$K_{1,\ell}$-free since no vertex has degree $\ell$. We now will show this construction is Berge-$K_{1,\ell}$-saturated. First, note that any edge $e$ we add must contain at least one vertex from $V\setminus C$, say $v\in V\setminus C$. Due to the linear hypergraph structure on $V\setminus C$, if $v\not\in D$, it is clear that before $e$ was added, $v$ was the center of a Berge-$K_{1,\ell-1}$, and that any choice of vertices from the $\ell-1$ edges incident with $v$ gives a legal choice for the core vertices of this Berge-$K_{1,\ell-1}$. Thus, as long as we first choose a vertex in $e\setminus v$ to be the core leaf in $e$ in the Berge-$K_{1,\ell}$ we are building, we always have a choice for a core leaf in the remaining edges incident with $v$ (note, $k\geq 3$, so each edge from the linear hypergraph has at least two choices for a core leaf, and only one could have been used already when we chose one). If $v\in D$, this is only slightly harder as we need to be careful about which vertex we choose to be the core leaf in the edge that contains $D$ and vertices from $C$. By choosing this core leaf from $C$, we guarantee this will not conflict with any vertices in the edges of the linear hypergraph, and so again we can proceed as before by choosing the core leaf in $e$, then choosing the remaining core leaves. Thus, adding $e$ has created a Berge-$K_{1,\ell}$, and thus the construction is saturated.
\end{proof}

\section{Conclusion}

We were able to determine exactly the saturation number of Berge stars. There are many other families of graphs that would be interesting to study though. The first family that comes to mind is complete graphs. The exact saturation numbers for Berge triangles were determined in \cite{EGKMS2017}, but not much is known about larger complete graphs. If $\ell\geq k+2$, then it can be seen that $\sat_k(n,K_\ell)\leq \binom{\ell-1}k(n-\ell+2)$. Indeed, given a vertex set $V$ with $|V|=n$, let $A\cup B=V$ be a partition with $|A|=\ell-2$. Then the hypergraph that contains every edge that intersects $B$ in at most one vertex is Berge-$K_\ell$-saturated. It is unclear if this is the optimal construction though, and this construction no longer works for $\ell\leq k+1$. 

Determining the saturation numbers for Berge cycles would also be very interesting. Some upper bounds on cycles are given in \cite{EGKMS2017}, but no non-trivial lower bounds are known.

\bibliographystyle{amsplain}
\bibliography{bib}{}

\end{document}